\documentclass{article}
\usepackage[english]{babel}
\usepackage{lmodern}
\usepackage[T1]{fontenc}
\usepackage[utf8]{inputenc}
\usepackage{amsmath}
\usepackage{graphicx}
\usepackage{esvect}
\usepackage{hyperref}
\usepackage{MnSymbol}
\usepackage{amsfonts}
\usepackage{float}
\usepackage{amsthm}
\usepackage{mathtools}
\usepackage{systeme} 
\usepackage{bbold}
\usepackage{relsize}
\usepackage{mathtools}
\usepackage{fancyhdr}
\usepackage{dcolumn}
\usepackage{authblk}
\usepackage[
backend=biber,
style=alphabetic,
sorting=ynt
]{biblatex}
\newcolumntype{2}{D{.}{}{2.0}}

\DeclareMathOperator{\rank}{\mathrm{rank}}

\addto\appendix{
	}
\newtheorem{theorem}{Theorem}[section]
\newtheorem{remark}{Remark}[section]
\newtheorem{defi}{Definition}[section]
\newtheorem{prop}{Proposition}[section]
\newtheorem{corollary}{Corollary}[section]

\numberwithin{equation}{section}

\let\oldref\ref
\renewcommand{\ref}[1]{(\oldref{#1})}

\usepackage[utf8]{inputenc}

\title{The geometry of Riemannian curvature radii}
\author[1,2]{Eugenio Bellini}
\affil[1]{SISSA, Via Bonomea 265, 34136 Trieste, Italy}
\affil[2]{Dipartimento di Matematica e Applicazioni, Università degli Studi di Milano-Bicocca, Via R. Cozzi 55, 20126 Milano, Italy;\newline E-mail: eugenio.bellini01@universitadipavia.it}

\date{}

\begin{document}
	
	\maketitle
		\begin{abstract}
	    We study the geometric structures associated with curvature radii of curves with values on a Riemannian manifold $(M,g)$. We show the existence of sub-Riemannian manifolds naturally associated with the curvature radii and we investigate their properties. In the particular case of surfaces these sub-Riemannian structures are of Engel type. The main character of our construction is a pair of global vector fields $f_1,f_2$, which encodes intrinsic information on the geometry of $(M,g)$. 
	\end{abstract}
	\tableofcontents
	\section{Introduction}\label{sec_1}
	\subsection{From contact elements to curvature radii}
The space of contact elements, first introduced by S. Lie (see \cite{Geiges} pages 6-11, or  \cite{Montgomery} pages 78-80) as a geometric tool for studying differential equations, marks the birth of contact geometry. Let us recall that given a $2$-dimensional Riemannian manifold $(M,g)$, the space of (oriented) contact elements of $M$ is the unit tangent bundle $UM$, i.e. the set of couples $(x,v)$ where $x\in M$, $v\in T_xM$ and $|v|=1$.
Any regular $M$-valued curve can be lifted to a $UM$-valued one through the maps
\begin{equation}\label{Lie_lift}
    \ell_{\pm}:\gamma \mapsto \left(\gamma,\pm\frac{\dot\gamma}{|\dot\gamma|}\right).
\end{equation}
There exists a contact distribution $\Delta$ over $UM$ naturally associated to the lifts \eqref{Lie_lift}. At a point $(x,v)\in UM$ such distribution is defined as 
\begin{equation*}
    \Delta_{(x,v)}=\langle\{v\}\rangle\oplus T_{v}(U_xM),
\end{equation*}
where we have denoted $U_xM=\{v\in T_xM\,:\,|v|=1\}$.
This distribution characterizes the images of \eqref{Lie_lift}, in the sense that a curve $(\gamma,v):[0,1]\to UM$ is in the image of one of such lifts if and only if it is tangent to $\Delta$ and $\gamma$ is a regular curve.
In this paper we study a second order generalization of the space of contact elements, in particular, given a Riemannian manifold $(M,g)$, we study the geometric structures associated to the lifts 
\begin{equation}\label{proj_curv_lift}
    c_{\pm}:\gamma\mapsto \left(\gamma, \pm\frac{\dot\gamma}{|\dot\gamma|}, R_g(\gamma)\right),
\end{equation}
where $R_g(\gamma)$ is the radius of curvature of $\gamma$ (the definition is recalled below \eqref{eq_rad_curv}).
Such second order construction has affinities with Cartan's prolongation of contact structures (see 6.3 of \cite{Montgomery}).
Recall that, given a Riemannian manifold $(M,g)$ (throughout the whole paper we assume $\dim(M)>1$), the geodesic curvature of a regular curve
$\gamma:[0,T]\to M$ is defined as
\begin{equation}\label{eq_geod_curv}
\kappa_{g}(\gamma):=\frac{|\pi_{\dot\gamma^\perp}\left(D_t\dot\gamma\right)|}{|\dot\gamma|^2},
\end{equation}
where $D_t$ denotes the covariant derivative along $\gamma$, and $\pi_{\dot\gamma^\perp}:T_{\gamma(t)}M\to T_{\gamma(t)}M$ denotes the orthogonal projection to $\{\dot\gamma(t)\}^\perp$.
If the geodesic curvature is never vanishing, we can define the radius of curvature of $\gamma$, computed with respect to $g$, as
\begin{equation}\label{eq_rad_curv}
R_g(\gamma):=\frac{1}{\kappa_g(\gamma)}\frac{\pi_{\dot\gamma^\perp}\left(D_t\dot\gamma\right)}{|\pi_{\dot\gamma^\perp}\left(D_t\dot\gamma\right)|}.
\end{equation}
With the sole purpose of simplifying the exposition and the expression of certain equations, we study the following modified versions of \eqref{proj_curv_lift}
\begin{equation}\label{2nd_ord}
   c_{\pm}:\gamma\mapsto \left(\gamma,\pm\frac{|R_g(\gamma)|}{|\dot\gamma|}\dot\gamma,R_g(\gamma)\right).
\end{equation}
We define the manifold of curvature radii of $(M,g)$, denoted with $\mathcal R(M,g)$, as the space of triples $(x,V,R)$ such that $x\in M$, $R\in T_x M\setminus \{0\}$ and $V \in\{R\}^\perp$, $|V|=|R|$. The map \eqref{2nd_ord} lifts regular $M$-valued curves with never vanishing geodesic curvature to $\mathcal R(M,g)$-valued ones.
The first central result of this work is the following theorem.
\begin{theorem}\label{thm_D}
		Let $(M,g)$ be a smooth n-dimensional Riemannian manifold, and let $\mathcal R(M,g)$ be the corresponding manifold of curvature radii. There exists a smooth, rank-n distribution, $\mathcal D(M,g)$, over $\mathcal R(M,g)$ with the property that a smooth curve $(\gamma,V,R):[0,1]\to \mathcal R(M,g)$ is in the image of one of the lifts \eqref{2nd_ord} if and only if it is tangent to $\mathcal D(M,g)$ and $\gamma$ is a regular curve. A local basis for $\mathcal D(M,g)$ is given by $n$ local vector fields $\{f_1,\dots,f_n\}$, which can be characterized in terms of the ODEs satisfied by their integral curves:
		\begin{equation}\label{eq_frame}
			f_1:\begin{cases}
				\dot x=V,\\
				D_t R=-V,\\
				D_t V=R,
			\end{cases}
			f_2:\begin{cases}
				\dot x=0,\\
				\dot R=R,\\
				\dot V=V,
			\end{cases}
			f_j:\begin{cases}
				\dot x=0,\\
				\dot R=b_j(x,V,R),\\
				\dot V=0,
			\end{cases}
		\end{equation}
		$j=3,\dots,n$, where $\left\{b_3(x,V,R),\dots,b_n(x,V,R)\right\}$ is a norm-$|R|$ local orthogonal basis of $\{R,V\}^{\perp}$. The distribution $\mathcal D(M,g)$ is bracket generating of step 3, it holds:
		\begin{equation}
				\begin{aligned}
					&\mathcal D(M,g)=\langle\{ f_1,\dots,f_n\}\rangle,\\
					&\mathcal D^2(M,g)=\mathcal D(M,g)\oplus \langle\{f_{1k}\}_{k=2}^n\rangle,\\
					&\mathcal D^3(M,g)=\mathcal D^2(M,g)\oplus \langle\{f_{1k1}\}_{k=2}^n\rangle=T\mathcal R(M,g).
				\end{aligned}
			\end{equation}
\end{theorem}
The fields $f_1,f_2$ described in equation \eqref{eq_frame} are closely related to the geometry of $(M,g)$. 
Indeed, as shown in Section \ref{secf1f2}, many geometric features of the original Riemannian manifold can be recovered considering their Lie brackets. The first bracket $[f_1,f_2]$ gives us back the geodesic flow of $(M,g)$ (see Proposition \ref{prop_Factor}); in a way these fields factorize the geodesic flow. Moreover we can read the Riemann curvature tensor in the structure constant of the frame \eqref{eq_frame} (see Proposition \ref{prop_riemann}).
\subsection{Metric structures}
The knowledge of the curvature radii of curves of a Riemannian manifold $(M,g)$, is enough to characterize the metric $g$ up to a homothetic  transformation, indeed, as shown in Section \ref{secMain}, two Riemannian manifolds having the same curvature radii are homothetic (see Definition \ref{same_R} for the precise statement).
\begin{theorem}\label{thm_R_info}
Two Riemannian manifolds $(M,g)$ and $(N,\eta)$ have the same curvature radii if and only if they are homothetic .
\end{theorem}
It is then natural to endow the distribution $\mathcal D(M,g)$ with a metric which is invariant under the action of the homothety group of $(M,g)$. In Section \ref{secMain} we show the existence of a family of metrics $\eta_{a,b}$ on $\mathcal D(M,g)$, parametrized by two real numbers $a$ and $b$, having this invariance property. The triple $(\mathcal R(M,g), \mathcal D(M,g),\eta_{a,b})$ is a sub-Riemannian manifold (\cite{Agrachev}, \cite{Montgomery}), which we denote with $\mathcal R_{a,b}(M,g)$.
For any $(\gamma, R, V)$ in the image of the lift \eqref{2nd_ord} the metric $\eta_{a,b}$ satisfies the following equation
\begin{equation}
\begin{aligned}
\left|\frac{d}{dt}(\gamma,V,R)\right|^2_{\eta_{a,b}}=a^2\frac{|\dot\gamma|^2}{|R|^2}+b^2\frac{|D_tR|^2}{|R|^2}.
\end{aligned}
\end{equation}
The central result regarding these metrics is stated in the following theorem.
\begin{theorem}\label{thm_sym}
Let $(M,g)$ be a Riemannian manifold, let $a,b\in\mathbb R,b>0$, and let $\mathcal R_{a,b}(M,g)$ be the corresponding sub-Riemannian manifold of curvature radii. The following map is a group isomorphism
\begin{equation}
\begin{aligned}\label{eq_group}
Homothety(M,g)&\to Isometry(\mathcal R_{a,b}(M,g))\\
\varphi&\mapsto (\varphi_\star\oplus\varphi_{\star})_{|\mathcal R(M,g)}.
\end{aligned}
\end{equation}
In particular, if $(N,\eta)$ is another Riemannian manifold, $\mathcal R_{a,b}(M,\eta)$ is isometric to  $\mathcal R_{a,b}(M,g)$ if and only $(N,\eta)$ and $(M,g)$ have the same curvature radii.
\end{theorem}
In Section \ref{secR2} we show that the sub-Riemannian manifolds $\mathcal R_{a,b}(\mathbb R^2,g_e)$, where $g_e$ is the standard Euclidean metric, are all isometric to left-invariant sub-Riemannian structures on the group orientation preserving homothetic transformations of $(\mathbb R^2,g_e)$, and we give a characterization of their geodesics. These geometries are related to the left-invariant sub-Riemannian structure on the group of rigid motions of $\mathbb R^2$  (\cite{bicycle1}, \cite{bicycle2}, \cite{bicycle3}, \cite{bicycle4}).
\subsection{Structure of the paper}
In Section \ref{secSurf} we describe the sub-Riemannian manifold of curvature radii in the $2$-dimensional case. In Section \ref{secMain} we generalize the constructions of Section $2$ to an arbitrary Riemannian manifold $(M,g)$ and we prove Theorems \ref{thm_D}, \ref{thm_R_info}, \ref{thm_sym}. In Section \ref{secf1f2} we show that taking Lie brackets of the fields $f_1,f_2$, mentioned in the abstract, we can reconstruct the geodesic flow of the original Riemannian manifold and its Riemann curvature tensor. Finally in Section \ref{secR2} we study the manifold of curvature radii of $\mathbb R^2$ endowed with an homothetic invariant metric.
	\section{Curvature Radii on surfaces}\label{secSurf}
	Let $(M,g)$ be a 2-dimensional Riemannian manifold, which for simplicity we assume to be oriented. Let $\gamma:[0,1]\to M$ be an arc-leength parametrized curve, then its curvature can be computed as 
\begin{equation}
    k_g(\gamma)=|D_t\gamma|,
\end{equation}
and, provided that $k_g(\gamma)$ is never vanishing, its radius of curvature is 
\begin{equation}
    R_g(\gamma)=\frac{1}{k_g(\gamma)^2}D_t\gamma.
\end{equation}
For every $t\in[0,1]$, the curvature radius of $\gamma$ at the point $\gamma(t)$ is a non-zero tangent vector, for this reason we define the \textit{manifold of curvature radii of $(M,g)$} as $\mathcal R(M)=TM\setminus s_o$, where $s_o$ is the zero-section of $TM$. Every regular curve with non-vanishing curvature can be lifted to a $\mathcal R(M)$-valued curve through the radius of curvature map 
	\begin{equation}\label{2dim-lift}
		\gamma\mapsto (\gamma, R_g(\gamma)).
	\end{equation}
    We would like to find a distribution $\mathcal D(M,g)$ over $\mathcal R(M)$ characterizing the image of such lift, in the sense that a curve $(\gamma, R):[0,1]\to \mathcal R(M)$ is in the image of \eqref{2dim-lift} if and only if it is tangent to $\mathcal D(M,g)$ and $\gamma$ is regular. To build such a distribution at a point $(x_0,R_0)\in \mathcal R(M)$ we collect all the velocities of all radii of curvature going through this point, and we take the vector space generated by them
	\begin{equation}\label{2dim_D}
		\mathcal D(M,g)_{(x_0,R_0)}:=\left\langle\left\{\frac{d}{dt}(\gamma,R_g(\gamma))(0)\,:\,\gamma(0)=x_0,\,R_g(\gamma)(0)=R_0\right\}\right\rangle.
	\end{equation}
	Since $M$ is oriented we have a complex strcuture
	\begin{equation}\label{perp_map}
		\begin{aligned}
			^{\perp_g}:TM\setminus s_o&\to TM \setminus s_o\\
			(x,R)&\mapsto (x,R^{\perp_g})
		\end{aligned}
	\end{equation}
	where $R^{\perp_g}$ is the unique vector orthogonal to $R$, positively oriented with it, satisfying $|R|_g=|R^{\perp_g}|_g$.
	\begin{prop}\label{prop_2dim}
		The collection of vector spaces defined in \eqref{2dim_D} is a smooth Engel distribution over $\mathcal R(M)$ (for basic facts on Engel distributions see for instance \cite{Bryant}, chapter 2). Moreover a curve $(\gamma,R):[0,T]\to \mathcal R(M)$ is in the image of the lift \eqref{2dim-lift}
		if and only if it is tangent to $\mathcal D(M,g)$ and $\gamma$ is a regular curve. A basis for $\mathcal D(M,g)$ is given by two vector fields $f^g_1,f^g_2$, which are characterized in terms of the ODEs satisfied by their integral curves as
		\begin{equation}\label{2dim_frame}
			f^g_1:\,\begin{cases}
				\dot x=R^{\perp_g},\\
				D_t R=-R^{\perp_g},
			\end{cases}
			f^g_2:\,\begin{cases}
				\dot x=0,\\
				\dot R=R,
			\end{cases}
		\end{equation}
		where $D_tR$ denotes the covariant derivative of $R(t)$ along the curve $x(t)$.
	\end{prop}
	\begin{proof}
		It is a special case of Theorem \ref{thm_D}, which is proved in Section \ref{secMain}.
	\end{proof}
	What is the geometric significance of the fields $f_1^g,f_2^g$?  It is quite explicit that the integral curves of $f_2^g$ are dialation of $R$ in a fixed fiber, i.e. curves of the kind $t\mapsto (x_0,e^tR_0)$, whereas if $(\gamma,R):[0,1]\to\mathcal R(M)$ is an integral curve of $f_1^g$, as a consequence of Proposition \ref{prop_2dim}, $R$ is the radius of curvature of $\gamma$, and hence $\kappa_g(\gamma)=1/|R|_g$. On the other hand according to \eqref{2dim_frame} $D_tR=-R^{\perp_g}$, therefore 
	\begin{equation*}
		\frac{d}{dt}|R|^2_g=-2\langle R,R^{\perp_g}\rangle_g=0,
	\end{equation*}
	hence $\kappa_g(\gamma)$ is constant. Thus the integral curves of $f_1^g$ are exactly lifts of curves with constant geodesic curvature.
	Homothetic transformations preserve the curvature radius map. 
   It is then natural to endow $\mathcal D(M,g) $ with a metric which is invariant under the action of the homothety group of $(M,g)$. For every $a,b\in\mathbb R$, $b\neq 0$, we define a metric $\eta^g_{a,b}$ on $\mathcal D(M,g)$ by declaring the fields 
	\begin{equation}\label{flat_frame_a,b2}
		\begin{aligned}
			&f_1^{g,a,b}:=\frac{1}{\sqrt{a^2+b^2}}f^g_1,\\
			&f_2^{g,a,b}:=\frac{1}{b}f^g_2,
		\end{aligned}
	\end{equation}
	an orthonormal frame. A simple computation, which is a particular case of \eqref{metric_computation}, shows that if $(\gamma,R)$ is an admissible curve, then the metric $\eta_{a,b}^g$ satisfies the following equation
	\begin{equation}
		\begin{aligned}
			\left|\frac{d}{dt}(\gamma,R)\right|^2_{\eta_{a,b}^g}=a^2\frac{|\dot\gamma|^2}{|R|^2}+b^2\frac{|D_tR|^2}{|R|^2}.
		\end{aligned}
	\end{equation}
The homogeneity of the right hand side shows the homothetic invariant nature of the metric $\eta^g_{a,b}$.
 The triple $(\mathcal R(M), \mathcal D(M,g), \eta_{a,b}^g)$ defines a sub-Riemannian manifold, which we denote with $\mathcal R_{a,b}(M,g)$.
	\section{The sub-Riemannian manifold of curvature radii}\label{secMain}
	We would like to extend the construction of $\mathcal R_{a,b}(M,g)$, presented for surfaces in Section \ref{secSurf}, to an arbitrary Riemannian manifold. One of the advantages of working with surfaces is that, given a regular curve $\gamma$ with never vanishing geodesic curvature, the direction of $\dot\gamma$ is uniquely determined by the radius of curvature. This is no more true when $\dim(M)\geq 3$, and we need to keep track of the velocity's direction in some way. For this reason we cannot  define the manifold of curvature radii as $TM\setminus s_o$ anymore. Instead we define it as a subset of $TM\oplus TM$.
	\begin{defi}
		Let $(M,g)$ be a Riemannian manifold. The manifold of curvature radii of $(M,g)$, denoted by $\mathcal R(M,g)$, is defined by
		\begin{equation}\label{R(M,g)}
			\mathcal R(M,g):=\{(x,V,R)\in TM\oplus TM\,:\,\langle R,V\rangle_g=0,\,|R|_g=|V|_g>0\}.
		\end{equation}
	\end{defi}
	From now on, when the metric $g$ we are referring to is clear from the context, we drop the $g$ subscript, for instance we denote $|\cdot|=|\cdot|_g$, $\perp=\perp_g$ and so on.
	One can check that, if $M$ is an $n$-dimensional manifold, then $\mathcal R(M,g)$ is a smooth embedded sub-manifold of $TM\oplus TM$ of dimension $3n-2$. Moreover $\mathcal R(M,g)$ has the structure of a $\mathbb S^{n-2}$ bundle over $TM\setminus s_o$, where the fiber at $(x,R)\in TM\setminus s_o $ is the sphere of radius $|R|$ contained in the plane $R^{\perp}$,
	\begin{equation}\label{sphere_fibers}
		\mathcal R(M,g)_{(x,R)}=\{V\in T_xM\,:\,V\perp R,\,\,|V|=|R|\,\}\simeq \mathbb S^{n-2}.
	\end{equation}
	As expected, if $n=2$ these spheres have dimension $0$, and $\mathcal R(M,g)$ reduces to a 0-dimensional fibration over $TM\setminus s_o$. If $M$ is a surface, oriented for simplicity, we have
	\begin{equation*}
	\mathcal R(M,g)\cong TM\setminus s_o\cupdot TM\setminus s_o.
	\end{equation*}
	In general $\mathcal R(M,g)$ has also the structure of a fiber bundle over $M$, whose fibers are 
	\begin{equation*}
		\mathcal R(M,g)_x=\{(V,R)\in T_xM\oplus T_xM\,:\,\langle R,V\rangle=0,\,|R|=|V|>0\,\}.
	\end{equation*}
	There exist two lifts which allows us to map regular curves $\gamma:[0,1]\to M$ with never vanishing geodesic curvature, to $\mathcal R(M,g)$-valued curves 
	\begin{equation}\label{general_lift}
		c_{\pm}:\gamma\mapsto \left(\gamma, \pm|R_g(\gamma)|\frac{\dot\gamma}{|\dot\gamma|}, R_g(\gamma)\right),
	\end{equation}
	where $R_g$ is the map defined in \eqref{eq_rad_curv}.
	We would like to find a distribution characterizing the image of the lift \eqref{general_lift}, i.e. a distribution  $\mathcal D(M,g)$ such that $(\gamma,V,R):[0,1]\to \mathcal R(M,g)$ is the lift of some curve $\gamma$ if and only if it is tangent to $\mathcal D(M,g)$, and $\gamma$ is regular. To construct this distribution at $q_0:=(x_0,V_0,R_0)\in\mathcal R(M,g)$ we collect the velocities of all such curves going through this point, and we take the vector space generated by them:
	\begin{equation}\label{def_dist}
		\mathcal D(M,g)_{q_0}:=\left\langle\left\{\frac{d}{dt}(\gamma, R,V)(0)\,:\,(\gamma,V,R)(0)=q_0, (\gamma,V,R)\,\text{as in}\,\eqref{general_lift}  \,\right\}\right\rangle.
	\end{equation}
	Before moving to the proof of Theorem \ref{thm_D}, let us fix the notation 
	\begin{equation*}
	    f_{i_1i_2\dots i_k}=[f_{i_1},[f_{i_2},[\dots f_{i_k}]]\dots],
	\end{equation*}
    to indicate the iterated brackets of the fields $f_1,\dots,f_n$.
   Moreover we will often make use of the abbreviation 
   \begin{equation}\label{eq_nota}
	X\partial_x=\sum_{i=1}^nX^i\partial_{x^i}.
	\end{equation}
   where $(x^1,\dots,x^n)$ are local coordinates and $X\in \mathbb R^n$ is the vector $X=(X^1,\dots,X^n)$.
		\begin{proof}(Theorem \ref{thm_D}) 
			The vectorfields in \eqref{eq_frame} are, by definition, local sections of $T(TM\oplus TM)$. We need to show that they are actually tangent to $\mathcal R(M,g)\subset TM\oplus TM$.
			Let $(x_0,V_0,R_0)\in\mathcal R(M,g)$ and let $(x,V,R)(t)$ be an integral curve of $f_i$ with initial point $(x_0,V_0,R_0)$. We have to show that $\langle V(t), R(t)\rangle= 0$ and that $|V(t)|=|R(t)|$ for each $t$ such that the flow of $f_i$ is defined. Let us start with $f_1$, any of its integral curves satisfies 
			\begin{equation*}
				\begin{cases}
					\dot x=V,\\
					D_t R=-V,\\
					D_t V=R.
				\end{cases}
			\end{equation*}
			Therefore we have 
			\begin{equation*}
				\begin{aligned}
					\frac{d}{dt}\langle R, V\rangle=|R|^2-|V|^2,\,\,\,\frac{d}{dt}|R|^2-|V|^2=-4\langle R, V\rangle,
				\end{aligned}
			\end{equation*}
			but $|R(0)|^2-|V(0)|^2=0=\langle R(0), V(0)\rangle$, because $(x,V,R)(0)\in\mathcal R(M,g)$, then from uniqueness of ODEs solutions we obtain $|V(t)|-|R(t)|\equiv\langle V(t), R(t)\rangle\equiv 0$. Let $(x, R, V)(t)$ be an integral curve of $f_2$ with initial point in $\mathcal R(M,g)$, then from \eqref{eq_frame} we compute 
			\begin{equation*}
				\frac{d}{dt}\left(|R|^2-|V|^2\right)=2\left(|R|^2-|V|^2\right),\,\,\,\,\frac{d}{dt}\langle R,V\rangle=2\langle R,V\rangle.
			\end{equation*}
			Since $(\gamma,V,R)(0)\in\mathcal R(M,g)$, we deduce $|V(t)|-|R(t)|=\langle V(t), R(t)\rangle\equiv 0$. Let $(\gamma, R, V)(t)$ be an integral curve of $f_j$, $j=3,\dots, n$, then from \eqref{eq_frame} we have 
			\begin{equation*}
				\frac{d}{dt}\left(|R|^2-|V|^2\right)=0,\,\,\,\,\frac{d}{dt}\langle R,V\rangle=0,
			\end{equation*}
			hence $f_j$ is tangent to $\mathcal R(M,g)$.
			\newline
			Now we claim that a curve $(\gamma,V,R):[0,1]\to \mathcal R(M,g)$ is in the image of one of the lifts \eqref{general_lift} if and only if it is tangent to $\langle\{ f_1,\dots,f_n\}\rangle$, and $\gamma$ is regular. Notice that, by definition of $\mathcal D(M,g)$, this would imply $\mathcal D(M,g)=\langle\{ f_1,\dots,f_n\}\rangle$. To prove our claim we have to show that $(\gamma,V,R)$ is in the image of one of the lifts \eqref{general_lift} if an only if there exist $n$ smooth functions, $u_1,\dots,u_n:[0,1]\to\mathbb R$ , with $u_1\neq 0$, such that $\frac{d}{dt}(\gamma, R,V)=u_1f_1+\dots+u_nf_n$, or in other words, according to \eqref{eq_frame}
			\begin{equation}\label{eq_fund_syst}
				\begin{cases}
					\dot\gamma=u_1V,\\
					D_t R=-u_1V+u_2R+u_3b_3+\dots+u_nb_n,\\
					D_t V=u_1R+u_2V.
				\end{cases}
			\end{equation}
			Assume that $\sigma:=(\gamma,V,R):[0,1]\to\mathcal R(M,g)$ is in the image of \eqref{general_lift}, then, by definition $R=R_g(\gamma)$ and $\dot\gamma=u_1 V$, for some never vanishing smooth function $u_1:[0,1]\to\mathbb R$. Since the condition $\dot\sigma\in \langle \{f_1,\dots,f_n\}\rangle$ is independent of the parametrization, without loss of generality we may assume $|\dot\gamma|=1$, i.e. $u_1=1/|R|$. Since $\gamma$ is arc-length parametrized, as a consequence of formulae \eqref{eq_rad_curv} and \eqref{eq_geod_curv}, we have
			\begin{equation*}
				D_t\dot\gamma=\frac{1}{|R|^2}R.
			\end{equation*}
			On the other hand, since $\dot\gamma=1/|R|V$, we have 
			\begin{equation}\label{eq_D_tV}
				D_t\dot\gamma=D_t\frac{V}{|R|}=\frac{d}{dt}\left(\frac{1}{|R|}\right)V+\frac{1}{|R|}D_tV,
			\end{equation}
			therefore
			\begin{equation*}
				D_tV=\frac{1}{|R|}R-|R|\frac{d}{dt}\left(\frac{1}{|R|}\right)V=u_1R+u_2V,
			\end{equation*}
			where we have denoted $u_2:=-|R|\frac{d}{dt}1/|R|=\frac{d}{dt}\log{|R|}$.
			It remains to compute the covariant derivative of $R$ along $\gamma$. Notice that $\{V(t),R(t),b_3(t),\dots,b_n(t)\}$ is a basis for $T_{\gamma(t)}M$ for every $t\in[0,1]$, therefore there exist $\lambda_1,\dots,\lambda_n:[0,1]\to\mathbb R$ smooth functions such that 
			\begin{equation*}
				D_tR=\lambda_1 V+\lambda_2 R+\lambda_3 b_3+\dots\lambda_n b_n.
			\end{equation*}
			To prove that \eqref{eq_fund_syst} holds we just have to show that $\lambda_1=-u_1$, $\lambda_2=u_2$.
			We begin with $\lambda_1$:
			\begin{equation*}
				\begin{aligned}
					\lambda_1=\langle D_t R, \frac{V}{|V|^2}\rangle=-\frac{1}{|R|^2}\langle R, D_t V\rangle
					=-\frac{1}{|R|^2}\langle R, u_1R\rangle=-u_1.
				\end{aligned}
			\end{equation*}
			Concerning the second coefficient $\lambda_2$ we have 
			\begin{equation*}
				\begin{aligned}
					\lambda_2&=\langle D_t R, \frac{R}{|R|^2}\rangle=\frac{1}{2|R|^2}\frac{d}{dt}|R|^2=\frac{1}{2}\frac{d}{dt}\log{|R|^2}=\frac{d}{dt}\log{|R|}=u_2.
				\end{aligned}
			\end{equation*}
			Conversely let $\sigma=(\gamma,V,R)$ be a curve satisfying \eqref{eq_fund_syst}, with $u_1\neq 0$. We want to show that $R=R_g(\gamma)$, i.e. that $R$ is the radius of curvature of $\gamma$. 
			Exploiting \eqref{eq_fund_syst} we compute
			\begin{equation}\label{pidotgamma}
				\pi_{\dot\gamma^{\perp}}D_t\dot\gamma=\pi_{\dot\gamma^{\perp}}(u_1D_tV+\dot u_1 V)=u_1^2 R,
			\end{equation}
			then, by definition \eqref{eq_rad_curv}, the radius of curvature of $\gamma$ satisfies
			\begin{equation}\label{Rcompuation}
				R_g(\gamma)=\frac{\pi_{\dot\gamma^\perp} D_t\dot\gamma}{\kappa_g(\gamma)|\pi_{\dot\gamma^\perp} D_t\dot\gamma|}=\frac{R}{\kappa_g(\gamma)|R|}.
			\end{equation}
			On the other hand combining \eqref{eq_geod_curv} with \eqref{pidotgamma}, and considering that $|R|=|V|$, we deduce \begin{equation}\label{kcomput}
				\kappa_g(\gamma)=  \frac{|\pi_{\gamma^\perp} D_t\dot\gamma|}{|\dot\gamma|^2}=\frac{u_1^2|R|}{u_1^2|V|^2}=\frac{1}{|R|}.
			\end{equation}
			Substituting \eqref{kcomput} into \eqref{Rcompuation} we obtain $R_g(\gamma)=R$, as requested.\newline
			We now show that $\mathcal D(M,g)$ is an equiregular bracket generating distribution of step 3. To do this we need to be more precise about the definition of the local fields $f_3,\dots,f_n$ in equation \eqref{eq_frame}, in particular given $(x,V,R)\in \mathcal R(M,g)$ we have to make a choice of a local basis for $\{R,V\}^\perp$. Let $\mathcal U\subset M$ be an open subset and let $E_1,\dots,E_n$ be a local orthonormal frame for $g$ on $\mathcal U$. For every $i\neq j$ we define the following two-form $\omega_{ij}\in\Omega^2(\mathcal U)$:
			\begin{equation*}
				\omega_{ij}(R,V):=\det
				\begin{pmatrix}
					\langle R,E_i\rangle_g & \langle V,E_i\rangle_g \\
					\langle R,E_j\rangle_g & \langle V,E_j\rangle_g
				\end{pmatrix}
			\end{equation*}
			and the following open subset of $TM\oplus TM$
			\begin{equation}
				\mathcal U_{ij}:=\left\{(x,V,R)\,:\, \omega_{ij}(R,V)
				\neq 0\right\}.
			\end{equation}
			By construction, the sets $\{V_{ij}:=\mathcal U_{ij}\cap\mathcal R(M,g)\}_{i\neq j}$, constitute an open cover for $\mathcal R(M,g)\cap T\mathcal U\oplus T\mathcal U$. Moreover the vectors \begin{equation*} \{R,V,E_3(x),\dots,\hat E_{i}(x),\dots,\hat E_{j}(x),\dots,E_n(x)\},
	        \end{equation*} are linearly independent for every $(x,V,R)\in V_{ij}$. We define 
			\begin{equation*}
				\{e_3(x,V,R),\dots,e_n(x,V,R)\}
			\end{equation*}
			to be the norm-$|R|$ basis of $\{R,V\}^\perp$ obtained by applying the Gram-Schmidt algorithm to the vectors 
	        \begin{equation*} \{R,V,E_3(x),\dots,\hat E_{i}(x),\dots,\hat E_{j}(x),\dots,E_n(x)\},
	        \end{equation*}
	        and, on $V_{ij}$, we set
	        \begin{equation*}
	      f_k:=e_k\partial_R=\sum_{\mu=1}^n e_k^\mu \partial_{R^\mu},\,\,\,k=3,\dots,n.
	        \end{equation*}
	         Defined in this way, the local fields $e_k$ satisfy two useful properties :
			\begin{equation}
				e_k(x,\lambda R,\lambda V)=\lambda e_k(x,V,R),
			\end{equation}
			for every $\lambda>0$ and 
			\begin{equation}
				e_k(x,\cos \theta R-\sin \theta V, \sin \theta R+\cos \theta V)=e_k(x,V,R),
			\end{equation}
			for every $\theta\in[0,2\pi]$.
			Let $(x^\mu)$ be local coordinates on $M$ and let $(x^\mu,R^\mu,V^\mu)$ be the local coordinates induced by $(x^\mu)$ on $TM\oplus TM$. Let $\Gamma^{\mu}_{\alpha\beta}$ be the corresponding Christoffel symbols of the metric $g$. 
			Given $X,Y\in\mathbb R^n$, we denote with $\Gamma(X,Y)$ the row vector
			\begin{equation}\label{nota_symbols}
		     \Gamma(X,Y)=\sum_{\alpha,\beta=1}^n(\Gamma_{\alpha\beta}^1X^\alpha Y^\beta,\dots, \Gamma_{\alpha\beta}^nX^\alpha Y^\beta),
		\end{equation}
		then, making use of the notation \eqref{eq_nota}, the vector fields $f_1,f_2$ read
			\begin{equation}\label{eqf1f2}
				\begin{aligned}
					f_1&=V\partial_x-\left(V+\Gamma(V,R)\right)\partial_R+\left(R-\Gamma(V,V)\right)\partial_V,\\
					f_2&=R\partial_R+V\partial_V,
				\end{aligned}
			\end{equation}
			and their commutator reads 
			\begin{equation}\label{f12}
				\begin{aligned}
					[f_2,f_1]=&[R\partial_R+V\partial_V,V\partial_x-\left(V+\Gamma(V,R)\right)\partial_R+\left(R-\Gamma(V,V)\right)\partial_V]\\
					=&-\Gamma(V,R)\partial_R+R\partial_V+V\partial_x-\left(V+\Gamma(V,R)\right)\partial_R -2\Gamma(V,V)\partial_V\\
					&+\left(V+\Gamma(V,R)\right)\partial_R-\left(R-\Gamma(V,V)\right)\partial_V\\
					=&\,V\partial_x-\Gamma(V,R)\partial_R-\Gamma(V,V)\partial_V.
				\end{aligned}
			\end{equation}
			We define the vector field $X_{12}$ as
			\begin{equation}\label{X12}
				X_{12}:=f_1-f_{21}=-V\partial_R+R\partial_V.
			\end{equation}
			We define 
			\begin{equation}\label{f121=X121}
				f_{121}:=[f_{12},f_1]=[X_{12},f_1],
			\end{equation}
			and we notice that this vector field satisfies 
			\begin{equation}\label{piM121}
				\pi^{M}_\star f_{121}=\pi^{M}_\star[X_{12},f_1]=[R\partial_V,V\partial_x]=R\partial_x.
			\end{equation}
			Finally, to span the whole tangent space, we need the following vector fields from the third layer
			\begin{equation*}
				f_{1k1}=[[f_1,f_k],f_1],\,\,\, k=3,\dots,n.
			\end{equation*}
			We start by calculating the field $f_{1k}:=[f_1,f_k]$. Let $e_k(t)=e_k(x(t),R(t),V(t))$, where $(x(t),R(t),V(t))$ is an integral curve of $f_1$. Then, according to \eqref{eq_frame}, we have 
			\begin{equation*}
				\begin{aligned}
					&\langle D_te_k,R\rangle =-\langle e_k,D_tR\rangle 
				=\langle e_k,V\rangle =0,\\
					&\langle D_te_k,V\rangle =- \langle e_k,D_tV\rangle =-\langle e_k,R\rangle=0,\\
					&F_{k}^i:=\frac{1}{|R|^2}\langle D_te_k,e_i\rangle.
				\end{aligned}
			\end{equation*}
			Therefore 
			\begin{equation}
				f_1(e_k)=\sum_{i=3}^nF_{k}^ie_i-\Gamma(V,e_k),  
			\end{equation}
			and hence we have 
			\begin{equation}\label{f1k}
				\begin{aligned}
					f_{1k}=[f_1,f_k]&=[V\partial_x-\left(V+\Gamma(V,R)\right)\partial_R+\left(R-\Gamma(V,V)\right)\partial_V,e_k\partial_R]\\
					&=f_1(e_k)\partial_R+\Gamma(V,e_k)\partial_R-e_k\partial_V\\
					&=\left(\sum_{i=3}^nF_{k}^ie_i-\Gamma(V,e_k)\right)\partial_R+\Gamma(V,e_k)\partial_R-e_k\partial_V\\
					&=\sum_{i=3}^nF_{k}^ie_i\partial_R-e_k\partial_V=\sum_{i=3}^nF_{k}^if_i-e_k\partial_V.
				\end{aligned}
			\end{equation}
			Observe that the field $f_{1k1}$ satisfies 
			\begin{equation}\label{f1k1}
				\pi_{\star}^Mf_{1k1}=\pi_{\star}^M[f_{1k},f_1]
				=\pi_{\star}^M\left[\sum_{i=3}^nF_{k}^ie_i\partial_R-e_k\partial_V,V\partial_x\right]=-e_k\partial_x.
			\end{equation}
			Now we claim that the following $3n-2$ vector fields
			\begin{equation*}
				f_1,\dots,f_n,f_{12},f_{13},\dots,f_{1n},f_{121},f_{131},\dots,f_{1n1},
			\end{equation*}
			are linearly independent. Notice that, since $X_{12}=f_1+f_{12}$, we may equivalently claim that 
			\begin{equation}\label{linearind}
				f_1,\dots,f_n,X_{12},f_{13},\dots,f_{1n},f_{121},f_{131},\dots,f_{1n1},
			\end{equation}
			are linearly independent.
			Consider a vanishing linear combination, with coefficients in $\mathcal C^\infty(\mathcal R(M,g))$, of \eqref{linearind}:
			\begin{equation}\label{eq_sumzero}
				\sum_{k=1}^na_kf_k+a_{12}X_{12}+\sum_{k=3}^na_{1k}f_{1k}+a_{121}f_{121}+\sum_{k=3}^n a_{1k1}f_{1k1}=0.
			\end{equation}
			If we apply $\pi^{M}_{\star}$ to both sides of \eqref{eq_sumzero}, since according to \eqref{eq_frame}, \eqref{X12}, \eqref{f1k}, the fields $f_2,\dots,f_n,X_{12},f_{13},\dots,f_{1n}$ have $x$-component equal to zero, we are left with
			\begin{equation*}
				\pi_\star^M\left(a_1f_1+a_{121}f_{121}+\sum_{k=3}^n a_{1k1}f_{1k1}\right)=0,
			\end{equation*}
			which thanks to \eqref{eqf1f2}, \eqref{piM121} and \eqref{f1k1} translates to 
			\begin{equation}
				a_1V+a_{121}R-\sum_{k=3}^n a_{1k1}e_k=0,
			\end{equation}
			therefore $a_1=a_{121}=a_{131}=\dots=a_{1n1}=0$ and equation \eqref{eq_sumzero} becomes
			\begin{equation}\label{eq_sumzero1}
				\sum_{k=2}^na_kf_k+a_{12}X_{12}+\sum_{k=3}^na_{1k}f_{1k}=0.
			\end{equation}
			If we compute the V-component of both sides of \eqref{eq_sumzero1}, according to \eqref{eqf1f2},\eqref{X12},\eqref{f1k} we get \begin{equation*}
				a_2V+a_{12}R-\sum_{k=3}^na_{1k}e_k=0,
			\end{equation*}
			thus $a_2=a_{12}=a_{123}=\dots=a_{12n}=0$, and equation \eqref{eq_sumzero1} reads
			\begin{equation}\label{eq_sumzero2}
				\sum_{k=3}^na_kf_k=0,
			\end{equation}
			but $f_3,\dots,f_n$ are linearly independent by definition, therefore also $a_3=\dots=a_n=0$.
			We have just proved that $\mathcal D^3(M,g)=T\mathcal R(M,g)$. It remains to show that $\text{rank}\,\mathcal D^2(M,g)< \text{rank}\,T\mathcal R(M,g)=3n-2$. Notice that 
			\begin{equation*}
				\mathcal D^2(M,g)=\mathcal D(M,g)+\langle \{[f_1,f_j]\}_{j=2}^n \rangle +\langle \{[f_i,f_j]\}_{i,j=2}^n \rangle.
			\end{equation*}
			We claim that the distribution $\langle\{ f_2,\dots,f_n\}\rangle$ is integrable. If that is the case then $\langle\{[f_i,f_j]\}_{i,j=2}^n \rangle$ $\subset$ $\mathcal D(M,g)$ and as a consequence 
			\begin{equation}\label{D^2}
				\mathcal D^2(M,g)=\mathcal D(M,g)+\langle \{[f_1,f_j]\}_{j=2}^n \rangle ,
			\end{equation}
			therefore, since $f_1,\dots,f_n,f_{12},\dots,f_{1n}$ are linearly independent, we deduce
			\begin{equation}
				\text{rank}\,\mathcal D^2(M,g)=\text{rank}\, \mathcal D(M,g)+\text{rank} \mathcal \langle \{[f_1,f_j]\}_{j=2}^n \rangle=2n-1<3n-2.
			\end{equation}
			To prove our claim observe that, for $j=2,\dots,n$, $[f_2,f_j]=0$, indeed $e^{tf_2}(x,V,R)=(x,e^tR,e^tV)$, hence by construction $e_j(e^{tf_2}(x,V,R))=e^te_j(x,V,R)$, or in other words $f_2(e_j)=e_j$, but then
			\begin{equation*}
				[f_2,f_j]=f_2(e_j)\partial_R-e_j\partial_R=0.
			\end{equation*}
			Now consider $f_j,f_i$ with $i,j\geq3$, to lighten the notation in the following we have denoted $e_{ij}(t)=e_j(e^{tf_i}(x_0,V_0,R_0))$, $V(t)=V(e^{tf_i}(x_0,V_0,R_0))$ and $R(t)=R(e^{tf_i}(x_0,V_0,R_0))$. Let us compute 
			\begin{equation*}
				\begin{aligned}
					&\langle \dot e_{ij}, R\rangle=-\langle e_{ij},\dot R\rangle =-\langle e_{ij},e_{ii}\rangle=0,\\
					&\langle \dot e_{ij}, V\rangle =-\langle e_{ij},\dot V\rangle =0,
				\end{aligned}
			\end{equation*}
			and let us denote $a_{ijk}=\langle \dot e_{ij}, e_k\rangle/|R|^2$, then we have \begin{equation*}
				f_j(e_i)=\dot e_{ij}=\sum_{k=3}^na_{ijk}e_k,
			\end{equation*}
			meaning that 
			\begin{equation*}
				[f_i,f_j]=(f_i(e_j)-f_j(e_i))\partial_R
				=\sum_{k=3}^n(a_{jik}-a_{ijk})e_k\partial_R=\sum_{k=3}^n(a_{jik}-a_{ijk})f_k,
			\end{equation*}
			and this concludes the proof of the claim.
			Equation \eqref{D^2}, together with the fact that the fields $\{f_j\}_{j=1}^n$,$\{f_{1k}\}_{k=2}^n$, $\{f_{1k1}\}_{k=3}^n$ are a basis of $T\mathcal R(M,g)$, gives us the following local characterization of $\mathcal D(M,g)$'s flag:
			\begin{equation}
				\begin{aligned}
					&\mathcal D(M,g)=\langle f_1,\dots,f_n\rangle,\\
					&\mathcal \mathcal \mathcal D^2(M,g)=\mathcal D(M,g)\oplus \langle\{f_{1k}\}_{k=2}^n\rangle,\\
					&\mathcal D^3(M,g)=\mathcal D^2(M,g)\oplus \langle\{f_{1k1}\}_{k=2}^n\rangle=T\mathcal R(M,g).
				\end{aligned}
			\end{equation}
			In particular we have 
			\begin{equation}
				\text{Growth}\,\mathcal D(M,g)=(n,2n-1,3n-2).
			\end{equation}
		\end{proof}
We obtain the following corollary.
  \begin{corollary}
  The following inclusions hold:
\begin{equation}\label{nil_field}
			    \begin{aligned}
			        &f_{jk}\in\mathcal D(M,g),\,\,\,j,k=2\dots,n, \\
			        &f_{j1k}\in\mathcal D^2(M,g),\,\,\,j,k=2,\dots,n.
			    \end{aligned}
			\end{equation}
  \end{corollary}
  \begin{proof}
      We have already proved the first inclusion of \eqref{nil_field} since we have shown that $\langle\{f_2,\dots,f_n\}\rangle$ is integrable. To prove the second inclusion observe that the kernel of $\pi^M_\star$ has dimension $(3n-2)-n=2n-2$. On the other hand equations \eqref{eq_frame},  \eqref{X12}, \eqref{f1k}, imply that the following $2n-2$ linearly independent local smooth sections of $\mathcal D^2(M,g)$
			\begin{equation}
			f_2,f_3,\dots,f_n,X_{12},f_{13},\dots,f_{1n},
			\end{equation}
			are contained in $\ker\pi_\star^M$, and hence they constitute a basis for it:
			\begin{equation*}
				\ker\pi_\star^M=\langle\left\{ f_2,f_3,\dots,f_n,X_{12},f_{13},\dots,f_{1n}\right\}\rangle.
			\end{equation*}
		The distribution $\ker\pi_\star^M$ is integrable: if $X,Y$ are sections of $\ker\pi_\star^M$, then they are $\pi^M$-related to the zero-section of $TM$, and therefore also their bracket is so. Thanks to the integrability of $\ker\pi_\star^M$, recalling that $X_{12}=f_1+f_{12}$, we deduce that $f_{j1k}\in\mathcal D^2(M,g),\,\,\,j,k=2,\dots,n$.
  \end{proof}
	If $(M,g)$ is a Riemannian surface, which for simplicity we assume to be orientable, then the fields $f_1,f_2$ defined in \eqref{eq_frame} can be related with the fields $f_1^g$, $f_2^g$ defined in \eqref{2dim_frame}. Indeed, in this case $\mathcal R(M,g)$ has two connected components:
	\begin{equation*}
	    \mathcal R(M,g)=\{(x,R,R^\perp)\,:\,(x,R)\in TM\setminus s_o\}\cupdot \{(x,R,-R^\perp)\,|\,(x,R)\in TM\setminus s_o\}.
	\end{equation*}
	The projection $p:\mathcal R(M,g)\to TM\setminus s_o$ restricted to the first of such components is a diffeomorphism satisfying  \begin{equation*}
		p_\star f_1=f^g_1,\,\,\,p_\star f_2=f^g_2.
	\end{equation*}
	This fact, combined with Theorem \ref{thm_D}, constitutes a proof of Propositions \ref{prop_2dim}.\newline
	Focusing back on the general case, we would like to define a metric on $\mathcal D(M,g)$ that preserves the symmetries of the curvature radii lift 
	\begin{equation}\label{curv_rad_lift}
		\begin{aligned}
			R_g: \{\gamma\in \mathcal C^2([0,1],M)\,:\,\gamma\,\, \text{regular},\,\kappa_g(\gamma)\neq 0\}&\to\mathcal C^{0}\left([0,1],TM
			\right)
			\\\gamma&\mapsto R_g(\gamma),
		\end{aligned}
	\end{equation}
	hence before defining such a metric, we need to have a better understanding of how much information is encoded in the mapping \eqref{curv_rad_lift}.
	\begin{defi}\label{same_R}
		Let $(M,g)$ and $(N,\eta)$ be Riemannian manifolds, we say that they have the same curvature radii if and only if there exists a diffeomorphism $\varphi:(M,g)\to (N,\eta)$ such that 
		\begin{equation}\label{sym_R}
			R_\eta\circ\varphi=\varphi_\star\circ R_g.
		\end{equation}
	\end{defi}
	\begin{remark}\label{rmk_curv}
	Equality \eqref{sym_R} is meant as an equality of maps, therefore the diffeomorphism $\varphi$ of Definition \ref{same_R} maps the domain of $R_g$ to the one of $R_\eta$. This implies that $\gamma:[0,1]\to M$ is the reparametrization of a geodesic of $(M,g)$ if and only if $\varphi\circ\gamma$ is the reparametrization of a geodesic of $(N,\eta)$. In particular 
	\begin{equation}\label{zero_curv}
		\kappa_g(\gamma)\equiv 0\iff \kappa_\eta(\varphi\circ\gamma)\equiv 0.
	\end{equation}
	\end{remark}
	Definition \ref{same_R} gives a precise meaning to the statement of Theorem \ref{thm_R_info}, which we now prove.
		\begin{proof}(Theorem \ref{thm_R_info})
			It is sufficient to prove the theorem for two Riemannian metrics $\eta,g$ on the same manifold $M$.
			Assume firs that $(M,g)$ and $(M,\eta)$ are homothetic manifolds, then $\eta=\lambda g$ for some $\lambda>0$. The two metrics have the same Levi-Civita connection $\nabla$, moreover for every $X,Y\in TM$,
			$X\perp_{g}Y$ if and only if $X\perp_{\eta}Y$, therefore we simply denote $\perp=\perp_{g}=\perp_{\eta}$.
			Let $\gamma:[0,T]\to M$ be a regular curve, then the curvature computed with respect to $\eta$ can be easily related to the curvature computed with respect to $g$:
			\begin{equation*}
				\kappa_\eta(\gamma)=\frac{|\pi_{\dot\gamma^\perp}\left(D_t\dot\gamma\right)|_{\lambda\,g}}{|\dot\gamma(t)|_{\lambda\,g}^2}=\frac{1}{\sqrt{\lambda}}\frac{|\pi_{\dot\gamma^\perp}\left(D_t\dot\gamma\right)|_{g}}{|\dot\gamma(t)|_{g}^2}=\frac{1}{\sqrt{\lambda}}\kappa_g(\gamma).
			\end{equation*}
			Therefore $k_\eta(\gamma)$ is never vanishing if and only if also $k_g(\gamma)$ is so, and in that case using equation \eqref{eq_rad_curv} we deduce $R_\eta(\gamma)=R_g(\gamma)$.\newline
			Conversely, assume that $(M,g)$ and $(M,\eta)$ have the same curvature radii, i.e. \eqref{sym_R} holds and 
			\begin{equation}\label{curvM}
				R_g(\gamma)=R_\eta(\gamma)
			\end{equation}
			for any regular curve $\gamma$ with never vanishing geodesic curvature.
			Observe that for any $X,Y\in T_xM$, $X\perp_g Y$ if and only if there exists $\gamma:[0,1]\to M$ such that $\gamma(0)=x$, $\dot\gamma(0)=X$ and $R_g(\gamma)(0)=Y$, but then, since $R_g=R_\eta$, we obtain $\perp_g=\perp_\eta$.
			According to Remark \ref{rmk_curv} we have
			\begin{equation*}
			D^g_t\dot\gamma\propto\dot\gamma\,\,\iff\,\, D^\eta_t\dot\gamma\propto\dot\gamma.
			\end{equation*}
			Hence for any curve $\gamma:[0,1]\to M$ the following proportionality relationship holds \begin{equation}\label{propto}
			D_t^g\dot\gamma-D_t^\eta\dot\gamma\propto\dot\gamma,
			\end{equation}
			and in particular 
			\begin{equation}\label{perp1}
			\pi_{\dot\gamma^\perp}D_t^g\dot\gamma=\pi_{\dot\gamma^\perp}D_t^\eta\dot\gamma.
			\end{equation}
			In light of \eqref{perp1}, condition \eqref{curvM} reduces to
			\begin{equation*}
				\kappa_g(\gamma)|\pi_{\dot\gamma^\perp}D_t^g\dot\gamma|_g=\kappa_\eta(\gamma)|\pi_{\dot\gamma^\perp}D_t^\eta\dot\gamma|_\eta,
			\end{equation*}
			which, by definition of geodesic curvature, in turn is equivalent to
			\begin{equation}\label{perp2}
				\frac{|\pi_{\dot\gamma^\perp}D_t^g\dot\gamma|_g}{|\dot\gamma|_g}=\frac{|\pi_{\dot\gamma^\perp}D_t^\eta\dot\gamma|_\eta}{|\dot\gamma|_\eta}.
			\end{equation}
			For any $X,Y\in T_xM$, $X$ is perpendicular to $Y$ if and only if there exists $\gamma:[0,1]\to M$ satisfying 
			\begin{equation*}
				\dot\gamma(0)=X,\,\,\,\pi_{\dot\gamma^\perp}D_t^g\dot\gamma(0)=Y.
			\end{equation*}
			Hence equation \eqref{perp2} implies that 
			\begin{equation}
			\frac{|Y|_g}{|X|_g}=\frac{|Y|_\eta}{|X|_\eta},\,\,\,\forall\,X\perp Y,
			\end{equation}
			which in turn implies that the two metrics are conformally related: there exists a smooth function $f:M\to\mathbb R$ such that 
			\begin{equation}
				\eta=e^{2f}g.
			\end{equation}
			Since the metrics are conformal, their Levi-Civita connection difference tensor can be easily computed
			\begin{equation*}
			T(X,Y):=\nabla^\eta_XY-\nabla^g_XY=X(f)Y+Y(f)X-\langle X,Y\rangle_g\,\text{grad}_g f.
			\end{equation*}
			On the other hand thanks to equation \ref{propto} we know that \begin{equation*}
				T(X,X)\propto X,
			\end{equation*}
			and hence for any $Y\perp X$ we have 
			\begin{equation*}
				0=\langle T(X,X),Y\rangle_g=-| X|^2_g\langle\text{grad}_g f,Y\rangle_g=-|X|^2_gdf(Y).
			\end{equation*}
			Since $X,Y$ are arbitrary orthogonal vectors, this implies that $f$ is constant.
		\end{proof}
	The curvature radius lift $R_g$ of a Riemannian manifold $(M,g)$ is a complete homothety invariant of the metric $g$. 
	It is then natural to define on $\mathcal D(M,g)$ a metric which is invariant under homothetic transformations. For every $a,b\in\mathbb R$, $b\neq 0$, we define a metric $\eta_{a,b}$ on $\mathcal D(M,g)$ by declaring the fields 
	\begin{equation}\label{flat_frame_a,b}
		\begin{aligned}
			&f_1^{a,b}:=\frac{1}{\sqrt{a^2+b^2}}f_1,\\
			&f_j^{a,b}:=\frac{1}{b}f_j,\,\,j=2,\dots,n,\\
		\end{aligned}
	\end{equation}
	a local orthonormal frame. 
	For any $(\gamma, R, V)$ in the image of the lift \eqref{general_lift} the metric $\eta_{a,b}$ satisfies the following equation
	\begin{equation}\label{eq_metricab}
		\begin{aligned}
			\left|\frac{d}{dt}(\gamma,V,R)\right|^2_{\eta_{a,b}}=a^2\frac{|\dot\gamma|^2}{|R|^2}+b^2\frac{|D_tR|^2}{|R|^2}.
		\end{aligned}
	\end{equation}
	To prove \eqref{eq_metricab} recall that according to Theorem \ref{thm_D} there exists $u_1,\dots,u_n:[0,1]\to\mathbb R$ such that
	\begin{equation}
		\frac{d}{dt}(\gamma,V,R)=u_1f^{a,b}_1+u_2f^{a,b}_2+\dots u_nf^{a,b}_n,
	\end{equation}
	or equivalently according to \eqref{eq_fund_syst}, 
	\begin{equation}
		\begin{aligned}
			\dot\gamma&=\frac{u_1}{\sqrt{a^2+b^2}}V,\\
			D_t R&=-\frac{u_1}{\sqrt{a^2+b^2}}V+\frac{u_2}{b}R+\frac{u_3}{b}e_3+\dots+\frac{u_n}{b}e_n,\\
			D_t V&=\frac{u_1}{\sqrt{a^2+b^2}}R+\frac{u_2}{b}V.
		\end{aligned}
	\end{equation}
	Hence
	\begin{equation}\label{metric_computation}
		\begin{aligned}
			a^2\frac{|\dot\gamma|^2}{|R|^2}+b^2\frac{|D_tR|^2}{|R|^2}=&a^2\frac{u_1^2}{|R|^2(a^2+b^2)}|R|^2+b^2\frac{u_1^2}{|R|^2(a^2+b^2)}|R|^2\\
			&+\frac{b^2}{|R|^2}\left(\frac{|R|^2}{b^2}u_2^2+\dots \frac{|R|^2}{b^2}u_n^2\right)\\
			=& u_1^2+u_2^2+\dots+u_n^2=\left|\frac{d}{dt}(\gamma,V,R)\right|^2_{\eta_{a,b}}.
		\end{aligned}
	\end{equation}
	The sub-Riemannian manifold $(\mathcal R(M,g),\mathcal D(M,g), \eta_{a,b})$ is called the \textit{sub-Riemannian manifold of curvature radii} of $(M,g)$ and it is denoted with $\mathcal R_{a,b} (M,g)$.
	We now prove the main result regarding these metrics, Theorem \ref{thm_sym}.
		\begin{proof}(Theorem \ref{thm_sym})
			We begin by making sure that the map \eqref{eq_group} is well defined. Let $\varphi:(M,g)\to (M,g)$ be a homothety of Riemannian manifolds. By construction $\varphi_\star\oplus\varphi_\star:TM\oplus TM\to TM\oplus TM$ maps $\mathcal R(M,g)$ to $\mathcal R(M,g)$, this is a simple consequence of the fact that $\varphi_\star$ preserves orthogonality and the ratios between norms of vectors. To prove that $\Phi:=\varphi_\star\oplus\varphi_\star$ is a sub-Riemannian isometry we show that $\Phi_\star f_i=f_i$, $i=1,2$, and that the fields $\{\Phi_\star f_3,\dots, \Phi_\star f_n\}$ are related to  $\{f_3,\dots,f_n\}$ by an orthogonal transformation.
			Let $(\gamma,V,R):[0,T]\to \mathcal R(M,g)$ be an integral curve of $f_1$ then we have 
			\begin{equation*}
				\begin{cases}
					\dot\gamma=V,\\
					D_t R=-V,\\
					D_t V=R,
				\end{cases}    
			\end{equation*}
			while the image under $\Phi$ of such curve, $\Phi(\gamma, R, V)=(\varphi\circ\gamma,\varphi_\star R, \varphi_\star V)=:(\bar\gamma,\bar R ,\bar V)$, satisfies 
			\begin{equation*}
				\begin{aligned}
		  &\dot{\bar\gamma}=\frac{d}{dt}\varphi(\gamma)=\varphi_\star\dot\gamma=\varphi_\star V=\bar V,\\
					&D_t\bar R=\nabla_{\bar V}\bar R=\varphi_\star\nabla_V R=-\varphi_\star V=-\bar V,\\
					&D_t\bar V=\nabla_{\bar V}\bar V=\varphi_\star\nabla_V V=\varphi_\star R=\bar R.
				\end{aligned}
			\end{equation*}
			Summarazing we have 
			\begin{equation*}
				\begin{cases}
					\dot{\bar\gamma}=\bar V,\\
					D_t\bar R=-\bar V,\\
					D_t \bar V=\bar R,
				\end{cases}
			\end{equation*}
			hence $\Phi_\star f_1=f_1$. The case of $f_2$ is analogous. Since $\varphi$ is a homothety it sends local orthogonal basis of $TM$ to local orthogonal basis, and if $X,Y\in T_xM$ have the same norm, then so do $\varphi_\star X$, $\varphi_\star Y$. Given $(x,V,R)\in\mathcal R(M,g)$, let $\{e_3(x,V,R),\dots,e_n(x,V,R)\}$ be the local orthogonal basis of $\{R,V\}^\perp$ constructed in the proof of Theorem \ref{thm_D}, then both $\{e_j(\varphi(x),\varphi_\star V,\varphi_\star R) \}_{j=3}^n$ and $\{\varphi_\star e_j(x,V,R)\}_{j=3}^n$ are othogonal basis of $\{\varphi_\star R,\varphi_\star V\}^{\perp}$, therefore they are related by an orthogonal transformation: there exists $O\in O(n-2)$ such that  
			\begin{equation}
			    \varphi_\star e_j(x,V,R)=O^i_je_i(\varphi(x),\varphi_\star V,\varphi_\star R).
			\end{equation}
			Now let $(\gamma,V,R):[0,1]\to\mathcal R(M,g)$ be an integral curve of $f_j$, for $j\geq 3$, we have 
			\begin{equation*}
			\begin{aligned}
		   &\frac{d}{dt}\varphi\circ\gamma=\varphi_\star\dot\gamma=0,\\
			    &\frac{d}{dt}\varphi_\star R=\varphi_\star \frac{d}{dt}R=\varphi_\star e_j(x,V,R)=O^i_je_i(\varphi(x),\varphi_\star V,\varphi_\star R),\\
			    &\frac{d}{dt}\varphi_\star V=\varphi_\star \frac{d}{dt}V=0.
			\end{aligned}
			\end{equation*}
			Hence the fields $\{\Phi_\star f_3,\dots, \Phi_\star f_n\}$ are related to  $\{f_3,\dots,f_n\}$ by an orthogonal transformation.\newline
			The map \eqref{eq_group} is an injective homomorphism by definition. We need to prove that it is also surjective. Let $\Phi:\mathcal R_{a,b}(M,g)\to\mathcal R_{a,b}(M,g)$, given by 
			\begin{equation}\label{PhiNote}
				\begin{aligned}
					(x,V,R)\mapsto\Phi(x,V,R)=(\Phi^x(x,V,R),\Phi^R(x,V,R),\Phi^V(x,V,R)),
				\end{aligned}
			\end{equation}
			be an isometry of sub-Riemannian manifolds, we have to show that there exists $\varphi:(M,g)\to (N,\eta)$ homothety of Riemannian manifolds such that
			\begin{equation*}
				\Phi=(\varphi_\star\oplus\varphi_\star)|_{\mathcal R(M,g)}.
			\end{equation*}
			Assume first $n>2$ and 
			consider the projection to $M$
			\begin{equation*}
				\begin{aligned}
					\pi^M:\mathcal R(M,g)&\to M\\
					(x,V,R)&\mapsto x.
				\end{aligned}
			\end{equation*}
			Notice that the kernel of $\pi^M_\star$ has dimension $(3n-2)-n=2n-2$. On the other hand, as shown in the proof of Theorem \ref{thm_D}, the equations \eqref{eq_frame}, \eqref{X12}, \eqref{f1k}, imply that the following $2n-2$ linearly independent smooth local sections of $\mathcal D^2(M,g)$
	        \begin{equation}\label{fieldKer}
				f_2,f_3,\dots,f_n,X_{12},f_{13},\dots,f_{1n},
			\end{equation}
			are contained in $\ker\pi_\star^M$, and hence they constitute a basis for it:
			\begin{equation*}
				\ker\pi_\star^M=\langle\left\{ f_2,f_3,\dots,f_n,X_{12},f_{13},\dots,f_{1n}\right\}\rangle.
			\end{equation*}
			We claim that $\ker\pi^M_\star\subset \mathcal D^2(M,g)$ is the unique integrable sub-bundle of $\mathcal D^2$ having maximal rank $2n-2$.
			We have already noticed that $\ker\pi_\star^M$ is integrable in the proof of Theorem \ref{thm_D}.
			Let $\mathcal D'\subset\mathcal D^2(M,g)$ be an integrable distribution and let $X,Y$ be smooth sections of $\mathcal D'$. Then there exists some smooth functions $\alpha_1,\dots,\alpha_{2n-1}$, $\beta_1,\dots,\beta_{2n-1}$ such that 
			\begin{equation*}
				\begin{aligned}
					X&=\alpha_1f_1+\alpha_2f_2+\dots+\alpha_{n+1}X_{12}+\alpha_{n+2}f_{13}+\dots+\alpha_{2n-1}f_{1n},\\
					Y&=\beta_1f_1+\beta_2f_2+\dots+\beta_{n+1}X_{12}+\beta_{n+2}f_{13}+\dots+\beta_{2n-1}f_{1n}.
				\end{aligned}
			\end{equation*}
			On the other hand, since $\mathcal D'\subset \mathcal D^2(M,g)$ is integrable, it holds
			\begin{equation}
				[X,Y]\equiv 0\mod{\mathcal D^2(M,g)},
			\end{equation}
			which translates to 
			\begin{equation*}
				\begin{aligned}
					(\alpha_1\beta_{n+1}-\beta_1\alpha_{n+1})[f_1,X_{12}]+\sum_{k=3}^n(\alpha_1\beta_{n+k-1}-\beta_1\alpha_{n+k-1})[f_1,f_{1k}]\equiv 0
				\end{aligned}
			\end{equation*}
			$\text{mod}\,\mathcal D^2(M,g)$. Since $[f_1,X_{12}]$, $[f_1,f_{1k}]$, $k=3,\dots,n$, are linearly independent mod$\,\mathcal D^2(M,g)$, we have
			\begin{equation}\label{coeff}
			\alpha_1\beta_{n+k-1}-\beta_1\alpha_{n+k-1}=0,\,\, k=2,\dots,n.
			\end{equation}
			If $\alpha_1\equiv\beta_1\equiv0$ then $X,Y\in\ker \pi^M_\star$. Otherwise if, for instance, $\alpha_1\neq0$, then $\beta_{n+k-1}=\frac{\beta_1}{\alpha_1}\alpha_{n+k-1}$ for $k=2,\dots,n$ and setting  
			\begin{equation*}
				Z=\alpha_{n+1}X_{12}+\alpha_{n+2}f_{13}+\dots+\alpha_{2n-1}f_{1n},
			\end{equation*}
			we deduce that 
			\begin{equation}
				X=\sum_{i=1}^n\alpha_if_i+Z,\,\,Y=\sum_{i=1}^n\beta_if_i+\frac{\beta_1}{\alpha_1}Z.
			\end{equation}
			As a consequence 
			\begin{equation*}
				\mathcal D'\subset\langle\{f_1,f_2,\dots,f_n,Z\}\rangle,
			\end{equation*}
			therefore $\rank\mathcal D'\leq n+1$. Observe that it is not possible that $\text{rank}\,\mathcal D'=n+1$, because otherwise $\mathcal D'$ would contain $\mathcal D(M,g)$, contradicting the hypothesis of integrability. We deduce that 
			\begin{equation*}
				\rank\mathcal D'\leq n<2n-2,\,\,\,\forall\,\,n>2,
			\end{equation*}
			proving our claim.
			Since $\Phi$ is an isometry, it preserves every layer of $\mathcal D(M,g)$'s flag, in the sense that 
			\begin{equation*}
				\Phi_\star \mathcal D^i(M,g)=\mathcal D^i(M,g),\,\, i=1,2,3,
			\end{equation*}
			therefore since $\ker\pi^M_\star\subset \mathcal D^2(M,g)$ is the unique integrable sub-bundle of $\mathcal D^2(M,g)$ having maximal rank $2n-2$, we deduce that 
			\begin{equation}
				\Phi_\star\ker\pi^M_\star=\ker\pi^M_\star,
			\end{equation}
			or equivalently 
			\begin{equation}\label{eqdefker}
				\ker(\pi^M\circ\Phi)_\star=\ker\pi^M_\star.
			\end{equation}
			Both $\ker\pi^M_\star$ and $\mathcal D(M,g)$ are invariant under $\Phi_\star$, therefore also their intersection is so, hence
			\begin{equation*}
				\Phi_\star\langle\{f_2,\dots,f_n\}\rangle=\Phi_\star(\ker\pi^M_\star\cap\mathcal D(M,g))=\ker\pi^M_\star\cap\mathcal D(M,g)=\langle\{f_2,\dots,f_n\}\rangle,
			\end{equation*}
			but then, since $\Phi_\star$ is a sub-Riemannian isometry preserving $\langle\{f_2,\dots,f_n\}\rangle$, it must preserve also its orthonormal complement, namely $\Phi_\star f_1=\pm f_1$.
			To show that $\Phi_\star f_2=\pm f_2$, consider the following linear map
			\begin{equation*}
				\begin{aligned}
					L:\ker \pi^M_{\star}\cap\mathcal D(M,g)&\to TM\\
					X&\mapsto \pi^M_\star[f_1,X],\\
				\end{aligned}
			\end{equation*}
			and notice that, according to \eqref{f12} and \eqref{f1k}, $\ker L=\langle \{f_3,\dots, f_n\}\rangle$.
			On the other hand, according to \ref{eqdefker} 
			\begin{equation*}
				\pi^M_\star[f_1,X]=0\iff \pi^M_\star\circ\Phi_\star[f_1,X]=0,
			\end{equation*}
			which, taking into account the fact that $\Phi_\star f_1=\pm f_1$, translates to
			\begin{equation*}
				\pi^M_\star[f_1,X]=0\iff \pi^N_\star[f_1,\Phi_\star X]=0,
			\end{equation*}
			or more simply to $\ker L=\Phi_\star\ker L$, thus $\Phi_\star\langle\{ f_3,\dots, f_n\}\rangle=\langle\{ f_3,\dots, f_n\}\rangle$. So far we have deduced that $\Phi_\star\langle\{ f_1,f_3,\dots, f_n\}\rangle=\langle\{f_1,f_3,\dots, f_n\}\rangle$, taking the orthonormal complement of this last equation with respect to the sub-Riemannian metric, we deduce that $\Phi_\star f_2=\pm f_2$.
			A direct consequence of equation \eqref{eqdefker} is that 
			\begin{equation}\label{finddiffeo}
				\Phi^x(x,V,R)=\Phi^x(x,R',V'),\,\,\,\forall\,\,(x,V,R),(x,R',V')\in\mathcal R(M,g),
			\end{equation}
			therefore the following map is well defined diffeomorphism
			\begin{equation*}
				\begin{aligned}
					\varphi:M&\to M\\
					x&\mapsto \Phi^x(x,V,R)=\pi^M\circ\Phi(x,V,R).
				\end{aligned}
			\end{equation*}
			Let $(\gamma,V,R)$ be an integral curve of $f_1$, then, since $\Phi_\star f_1= \pm f_1$, it holds
			\begin{equation}\label{eq_Phif1}
				\begin{cases}
					\dot \Phi^x=\pm\Phi^V,\\
					D_t\Phi^R=\mp\Phi^V,\\
					D_t\Phi^V=\pm\Phi^R,
				\end{cases}
			\end{equation}
			but we know that $\Phi^x=\varphi$, hence the first equation of \eqref{eq_Phif1} reads
			\begin{equation*}
				\varphi_\star V=\varphi_\star\dot\gamma= \dot\Phi^x=\pm\Phi^V,
			\end{equation*}
			consequently $\Phi^V=\pm\varphi_\star V$.
			Now consider the vector field $f_{21}=[f_1,f_1]$, recall that, according to \eqref{f12}, its integral curves satisfy 
			\begin{equation}
				\begin{cases}
					\dot x= V,\\
					D_tR=0,\\
					D_tV=0,
				\end{cases}
			\end{equation}
			on the other hand since $\Phi_\star f_1= \pm f_1$, $\Phi_\star f_2=\pm f_2$, it also holds $\Phi_\star f_{12}=\pm f_{12}$ thus
			\begin{equation*}
				\nabla_VV=0\iff \nabla_{\varphi_\star V}\varphi_\star V=0
			\end{equation*}
			or equivalently  
			\begin{equation}\label{eq_pullcone}
				\nabla=\varphi^\star\nabla.
			\end{equation}
			Exploiting \eqref{eq_pullcone} together with the third equation in \eqref{eq_Phif1} we deduce
			\begin{equation*}
				\pm\Phi^R=D_t\Phi^V=\nabla_{\varphi_\star V}\varphi_\star V=\varphi_\star \nabla_VV=\varphi_\star D_tV=\varphi_\star R,
			\end{equation*}
			from which we conclude 
			\begin{equation}\label{Phipm}
				\Phi(x,V,R)=(\varphi(x),\pm\varphi_\star R,\pm\varphi_\star V)=\pm \varphi_\star\oplus\varphi_\star(x,V,R).
			\end{equation}
			Substituting \eqref{Phipm} in the second equation of \eqref{eq_Phif1} we find that we have to discard the minus sign and we are left with
			\begin{equation}
				\Phi(x,V,R)=(\varphi(x),\varphi_\star V,\varphi_\star R).
			\end{equation}
			Let $(\gamma,V,R):[0,1]\to\mathcal R(M,g)$ be a curve tangent to $\mathcal D(M,g)$ with $\gamma$ regular, then according to Theorem \ref{thm_D}, $R=R_g(\gamma)$. On the other hand since $\Phi$ preserves $\mathcal D(M,g)$, also $(\varphi\circ\gamma,\varphi_\star R,\varphi_\star V)$ is tangent to $\mathcal D(M,g)$ and hence, according to Theorem \ref{thm_D} we have 
			\begin{equation*}
				R_g(\varphi\circ\gamma)=\varphi_\star R_g(\gamma).
			\end{equation*}
			On the other hand, since $\varphi:(M,\varphi^\star g)\to (M,g)$ is an isometry, it holds
			\begin{equation*}
			    R_g(\varphi\circ\gamma)=\varphi_\star R_{\varphi^\star g}(\gamma),
			\end{equation*}
			therefore $R_g=R_{\varphi^\star g}$ and, according to Theorem \ref{thm_R_info}, $\varphi :(M,g)\to (M,g)$ is a homothety.\newline
			Assume now that $n=\dim{M}=2$.
			Consider the following linear map
			\begin{equation}\label{Engel_map}
				\begin{aligned}
					P:\mathcal D(M,g)&\to\mathcal D^3(M,g)/\mathcal D^2(M,g)\\
					X&\mapsto \left[X,[f_1,f_2]\right]\mod\mathcal D^2(M,g).
				\end{aligned}
			\end{equation}
			Observe that the kernel of \eqref{Engel_map}, is invariant under pushforwards of sub-Riemannian isometries, in the sense that
			\begin{equation*}
				\Phi_\star\ker P=\ker P,
			\end{equation*}
			therefore
			\begin{equation*}
				\Phi_\star\langle\{f_2\}\rangle=\langle\{f_2\}\rangle,
			\end{equation*}
			and since $\Phi$ is an isometry we deduce that
			\begin{equation}
				\Phi_\star f_1=\pm f_1,\,\,\,\Phi_\star f_2=\pm f_2.
			\end{equation}
			We claim that $\Phi_\star f_2=f_2$. Assume by contradiction that $\Phi_\star f_2=-f_2$ and consider the following couple of vector fields 
			\begin{equation*}
				X=f_1+f_{12},\,\,\,Y=f_2.
			\end{equation*}
			Since $[f_2,f_{12}]=f_{12}$, we have 
			\begin{equation*}
				[X,Y]=[f_1,f_2]-f_{12}=f_{12}-f_{12}=0.
			\end{equation*}
			On the other hand since $\Phi_\star f_1=\pm f_1$, $\Phi_\star f_2=-f_2$ we find 
			\begin{equation*}
				0=\Phi_\star[X,Y]=[\pm f_1\mp f_{12},-f_2]=\mp 2f_{12}\neq 0,
			\end{equation*}
			which is a contradiction. Combining $\Phi_\star f_2=f_2$ and $\Phi_\star f_1=\pm f_1$ we deduce that
			\begin{equation*}
				\Phi_\star X_{12}=\pm X_{12},
			\end{equation*}
			therefore, since 
			\begin{equation*}
				\ker\pi^M_\star =\langle\{f_2,X_{12}\}\rangle,
			\end{equation*}
			we obtain 
			\begin{equation*}
				\ker\pi^M_\star =\ker(\pi^M\circ\Phi)_\star.
			\end{equation*}
			The remaining part of the proof is identical to the one of the case $n>2$. In particular we can conclude by repeating verbatim what is written between equation \eqref{finddiffeo} and the discussion of the $2$-dimensional case. 
		\end{proof}
	\section{The fields $f_1,f_2$}\label{secf1f2}
	Given a Riemannian manifold $(M,g)$, the vector fields $f_1,f_2$ defined in  \eqref{eq_frame} are global sections of $T(TM\oplus TM)$, which restrict to vector fields on $\mathcal R(M,g)$. As a consequence of Theorem \ref{thm_D} we know that they are metric invariants of $(M,g)$.  Actually they are even invariant under homothetic transformation. In the current section we show that some classical metric invariants can be recovered from the iterated Lie brackets of $f_1,f_2$. 
	Indeed, as the next proposition shows, the field $[f_1,f_2]$ already gives us a complete description of the geodesics of $(M,g)$; in this sense the fields $f_1,f_2$ give us a factorization of the geodesic flow.
	\begin{prop}{}\label{prop_Factor}
		Let $(M,g)$ be a Riemannian manifold, let $x\in M$ and let $\exp_x^{(M,g)}$ be the corresponding exponential map at $x$. Let $f_{21}=[f_2,f_1]$ and $f_{121}=[f_1,f_{21}]$, then, for every $(x,V,R)\in TM\oplus TM$, the following formulae hold
		\begin{equation}
			\begin{aligned}
				&\pi^M\circ e^{tf_{21}}(x,V,R)=\exp_x^{(M,g)}(tV),\\
				&\pi^M\circ e^{tf_{121}}(x,V,R)=\exp_x^{(M,g)}(tR),
			\end{aligned}
		\end{equation}
	for every $t\in\mathbb R$ such that the flow of the fields is defined.
	\end{prop}
	\begin{proof}
		According to equation \eqref{f12}, the vector field $f_{21}$ reads
		\begin{equation*}
			f_{21}=[f_2,f_1]=V\partial_x-\Gamma(V,R)R\partial_R-\Gamma(V,V)\partial_V,
		\end{equation*}
		where the symbols $\Gamma$ are the ones defined in equation \eqref{nota_symbols}.
		Consequently, any integral curve of $f_{21}$ satisfies 
		\begin{equation}\label{factor1}
			\begin{cases}
				\dot x=V,\\
				D_t R=0,\\
				D_tV=0.
			\end{cases}
		\end{equation}
		Therefore the curve $x(t)=\pi^M\circ e^{t[f_2,f_1]}(x_0,V_0,R_0)$ is the unique geodesic with initial point $x_0$ and initial velocity $V_0$.\newline
		Recall the vector field $X_{12}$ defined in equation \eqref{X12} and observe that 
		\begin{equation*}
		f_{121}=[f_1,f_{21}]=[f_1,f_1-X_{12}]=-[f_1,X_{12}],
		\end{equation*}
		therefore 
		\begin{equation*}
			\begin{aligned}
				f_{121}=&[V\partial_x-(V+\Gamma(R,V))\partial_R+(R-\Gamma(V,V))\partial_V, V\partial_R-R\partial_V]\\
				=&[V\partial_x-\Gamma(R,V)\partial_R-\Gamma(V,V)\partial_V, V\partial_R-R\partial_V]\\
				=&\Gamma(R,V)\partial_V-\Gamma(V,V)\partial_R+\Gamma(V,V)\partial_R+R\partial_x\\
				&-\Gamma(R,R)\partial_R-2\Gamma(R,V)\partial_V\\
				=&R\partial_x-\Gamma(R,R)\partial_R-\Gamma(R,V)\partial_V.
			\end{aligned}
		\end{equation*}
		Hence any integral curve of $f_{121}$ satisfies
		\begin{equation}\label{-f121}
			\begin{cases}
				\dot x=R,\\
				D_t R=0,\\
				D_tV=0,
			\end{cases}
		\end{equation}
	concluding the proof. 
	\end{proof}
	If we consider another layer of the Lie algebra generated by $f_1,f_2$, the components of Riemann curvature tensor appear.
	\begin{prop}
		Let $(M,g)$ be a Riemannian manifold, the integral curves of the vector field $f_{1121}:=[f_1,f_{121}]$ satisfy 
		\begin{equation}\label{Riemann}
			\begin{cases}
				\dot x=-V,\\
				D_t R=-\mathcal R^{\nabla}(V,R)R,\\
				D_t V=-\mathcal R^{\nabla}(V,R)V,
			\end{cases}
		\end{equation}
		where $\mathcal R^{\nabla}:\mathfrak X(M)\times\mathfrak X(M)\times\mathfrak X(M)\to\mathfrak X(M)$ is the Riemann curvature tensor of $(M,g)$. 
	\end{prop}
	\begin{proof}
		Given a vector field $X$ we denote with $X^x$ its $x$-component, with $X^R$ its $R$-component and with $X^V$ its $V$-component, meaning that 
		$X=X^x\partial_x+X^R\partial_R+X^V\partial_V$. We compute \eqref{Riemann} component by component:
		\begin{equation}\label{x-comp}
			\begin{aligned}
				f_{1121}^x&=f_1(f_{121}^x)-f_{121}(f_1^x)=f_1(R)-f_{121}(V)\\
				&=-V-\Gamma(R,V)+\Gamma(R,V)=-V.
			\end{aligned}
		\end{equation}
		Concerning the $R$-component we compute
		\begin{equation*}
			\begin{aligned}
				f_1(f_{121}^R)&=(V\partial_x-(V+\Gamma(R,V))\partial_R+(R-\Gamma(V,V))\partial_V)(-\Gamma(R,R))\\
				&=-V\partial_x\Gamma(R,R)+2\Gamma(V,R)+2\Gamma(\Gamma(R,V),R),
			\end{aligned}
		\end{equation*}
		and
		\begin{equation*}
			\begin{aligned}
				f_{121}(f_1^R)&=(R\partial_x-\Gamma(R,R)\partial_R-\Gamma(R,V)\partial_V)(-V-\Gamma(R,V))\\
				&=-R\partial_x\Gamma(R,V)+\Gamma(\Gamma(R,R),V)+\Gamma(R,V)+\Gamma(R,\Gamma(R,V)),
			\end{aligned}
		\end{equation*}
		therefore 
		\begin{equation}\label{R-comp}
			\begin{aligned}
				f_{1121}^R=&f_1(f_{121}^R)-f_{121}(f_1^R)\\
				=&\Gamma(R,V)-V\partial_x\Gamma(R,R)+R\partial_x\Gamma(R,V)\\
				&-\Gamma(V,\Gamma(R,R))+\Gamma(R,\Gamma(R,V))\\
				=&\Gamma(R,V)-\mathcal R^{\nabla}(V,R)R.
			\end{aligned}
		\end{equation}
		The quantity $f_1(f_{121}^V)$ can be computed as
		\begin{equation*}
			\begin{aligned}
				f_1(f_{121}^V)=&(V\partial_x-(V+\Gamma(R,V))\partial_R+(R-\Gamma(V,V))\partial_V)(-\Gamma(R,V))\\
				=&-V\partial_x\Gamma(R,V)+\Gamma(V,V)+\Gamma(\Gamma(R,V),V)-\Gamma(R,R)+\Gamma(R,\Gamma(V,V),
			\end{aligned}
		\end{equation*}
		whereas $f_{121}(f_1^V)$ satisfies
		\begin{equation*}
			\begin{aligned}
				f_{121}(f_1^V)&=(R\partial_x-\Gamma(R,R)\partial_R-\Gamma(R,V)\partial_V)(R-\Gamma(V,V))\\
				&=-R\partial_x\Gamma(V,V)-\Gamma(R,R)+2\Gamma(\Gamma(R,V),V),
			\end{aligned}
		\end{equation*}
		hence
		\begin{equation}\label{V-comp}
			\begin{aligned}
				f_{1121}^V=&f_1(f_{121}^V)-f_{121}(f_1^V
				)\\
				=&\Gamma(V,V)-V\partial_x\Gamma(R,V)+R\partial_x\Gamma(V,V)\\
				&-\Gamma(V,\Gamma(R,V))+\Gamma(R,\Gamma(V,V))\\
				=&\Gamma(V,V)-\mathcal R^{\nabla}(V,R)V.
			\end{aligned}
		\end{equation}
		Equations \eqref{x-comp}, \eqref{R-comp} and \eqref{V-comp} together imply that any integral curve of $f_{1121}$ satisfies the ODEs system \eqref{Riemann}.
	\end{proof}
	As stated in Theorem \ref{thm_D}, the fields
	\begin{equation}\label{eq_basis}
	    \{f_1,\dots,f_n,f_{12},\dots,f_{1n},f_{121},\dots,f_{1n1}\}
	\end{equation}
	constitute a local basis for $T\mathcal R(M,g)$, therefore there exists some real valued smooth functions $\{c_1,\dots,c_{3n-2}\}$ on $\mathcal R(M,g)$ such that
	\begin{equation}\label{eq_strct}
	    f_{1121}=\sum_{i=1}^nc_if_i+\sum_{i=2}^nc_{n+i-1}f_{1i}+\sum_{i=2}^nc_{2n+i-2}f_{1i1}.
	\end{equation}
	\begin{prop}\label{prop_riemann}
	Let $(M,g)$ be a Riemannian manifold, then the function $c_1:\mathcal R(M,g)\to \mathbb R$ defined in equation \eqref{eq_strct} is a homothetic invariant of $(M,g)$, which can be expressed in terms of sectional curvatures as
	\begin{equation}
	    c_1(x,V,R)=|R|^2\sec(R,V).
	\end{equation}
	In particular if $(M,g)$ is a Riemannian surface, then 
	\begin{equation}
	    c_1(x,V,R)=|R|^2K(x),
	\end{equation}
	where $K$ is the Gaussian curvature of $(M,g)$.
	\end{prop}
	\begin{proof}
	    From equations \eqref{Riemann} and \eqref{factor1} it follows that 
	    \begin{equation}
	        f_{1121}+f_{21}=-\mathcal R^\nabla(V,R)R\partial_R-\mathcal R^\nabla(V,R)V\partial_V,
	    \end{equation}
	    therefore 
	    \begin{equation*}
	        \begin{aligned}
	            f_{1121}+f_{21}=&-\frac{1}{|R|^2}\langle\mathcal R^\nabla(V,R)R,V\rangle V\partial_R-\frac{1}{|R|^2}\langle\mathcal R^\nabla(V,R)V,R\rangle R\partial_V\\
	            &-\frac{1}{|R|^2}\sum_{i=3}^n\left(\langle\mathcal R^\nabla(V,R)R,e_j\rangle e_j\partial_R +\langle\mathcal R^\nabla(V,R)V,e_j\rangle e_j\partial_V\right),\\
	            =&|R|^2\sec(R,V)(-V\partial_R+R\partial_V)-X,
	        \end{aligned}
	    \end{equation*}
	    where $\{e_3,\dots,e_n\}$ is the basis of $\{R,V\}^\perp$ described in the proof of Theorem \ref{thm_D} and we have denoted 
	    \begin{equation*}
	        X=\frac{1}{|R|^2}\sum_{i=3}^n\left(\langle\mathcal R^\nabla(V,R)R,e_j\rangle e_j\partial_R +\langle\mathcal R^\nabla(V,R)V,e_j\rangle e_j\partial_V\right).
	    \end{equation*}
	    Observe that 
	    \begin{equation*}
	        -V\partial_R+R\partial_V=f_1-f_{21},
	    \end{equation*}
	    hence 
	    \begin{equation}
	        f_{1121}+f_{21}=|R|^2\sec(R,V)(f_1-f_{21})-X.
	    \end{equation}
	The unique expression of the vector field $X$ in terms of the basis \eqref{eq_basis} is a linear combination which does not involve $f_1$, hence we deduce 
	\begin{equation}
	    c_1(x,V,R)=|R|^2\sec(R,V).
	\end{equation}
	\end{proof}
	The fields $f_1,f_2$ allow us to characterize the homotheties of Riemannian surfaces with one synthetic equation. Let $M$ be a smooth manifold and let $X\in\mathfrak X(M)$ be a vector field, which we assume to be complete for simplicity. The family of maps $\{e^{tX}_\star\}_{t\in\mathbb R}$ constitutes a one-parameter group of diffeomorphisms of $TM$; we denote its infinitesimal generator with $\vv{X}$. 
	\begin{prop}\label{char_hom}
		Let $(M,g)$ be a Riemannian surface and let $f_1$ be the local vector vield over $TM\setminus s_o$ defined in \eqref{2dim_frame}. A vector field $X\in\mathfrak X(M)$ is the infinitesimal generator of a one-parameter group of homotheties if and only if 
		\begin{equation}\label{car1}
			[\vv{X},f_1]=0.
		\end{equation}
	\end{prop}
	\begin{proof}
		A vector field $X\in\mathfrak X(M)$ satisfies \eqref{car1} if and only if 
		\begin{equation}
			e^{t\vv X}_\star f_1=f_1
		\end{equation}
		for every $t\in \mathbb R$. On the other hand, any vector field $X\in\mathfrak X(M)$ satisfies 
		\begin{equation}
			e^{t\vv X}_\star f_2=f_2,
		\end{equation}
		indeed 
		\begin{equation*}
			\begin{aligned}
				e^{s\vv X}\circ e^{tf_2}(x,R)&=e^{s\vv X}(x,e^t R)=(e^{sX}x,e^{sX}_\star e^tR)\\
				&=(e^{sX}x,e^te^{sX}_\star R)=e^{tf_2}\circ e^{s\vv X}(x,R).
			\end{aligned}
		\end{equation*}
		Therefore $X$ satisfies \eqref{car1} if and only if $e^{t\vv X}$ is an isometry of $\mathcal R_{a,b}(M,g)$ for each $t\in\mathbb R$, hence, by Theorem \ref{thm_sym}, $X$ satisfies \eqref{car1} if and only if $e^{tX}$ is a one-parameter group of homotheties of $(M,g)$. 
	\end{proof}
    The results obtained so far can be used to produce a flatness theorem for Riemannian surfaces having a $4$-dimensional Lie algebra of homothetic vector fields.
    \begin{theorem}\label{flat_thm}
   Let $(M,g)$ be a 2-dimensional Riemannian manifold and let $L\subset\mathfrak X(M)$ be the corresponding Lie algebra of homothetic vector fields. Then $\dim(L)\leq 4$, and $(M,g)$ is flat if equality is achieved.
\end{theorem}
\begin{proof}
  Let $\mathcal L\subset \mathfrak X(\mathcal R(M))$
   be the Lie algebra of isometric vector fields for the manifold of curvature radii $\mathcal R(M,g)$. By Theorem \ref{thm_sym} this Lie algebra is isomorphic to the one of homothetic vector fields of $(M,g)$. Let $X\in\mathcal L$ be a vector field vanishing at some point $(x, R)\in\mathcal R(M)$. We claim that $X$ is identically zero. Indeed, since $X$ is isometric, by Proposition \ref{char_hom}, it satisfies $[X,f_1]=[X,f_2]=0$.
   Since $\mathcal D=span\{f_1,f_2\}$ is bracket generating, for any $(x',R')\in\mathcal R$ there exist some real numbers $s_1,\dots,s_k$ such that 
   \begin{equation}
       (x',R')=e^{s_1 f_{i_1}}\circ\dots\circ e^{s_k f_{i_k}}(x,R),\,\,\,\,i_1,\dots,i_k\in\{1,2\},
   \end{equation}
consequently
   \begin{equation*}
       X_{(x',R')}=e^{s_1 f_{i_1}}_\star\circ\dots\circ e^{s_k f_{i_k}}_\star X_{(x,R)}=0.
   \end{equation*}
    Let $X_1,\dots,X_5\in\mathcal L$ and $(x,R)\in\mathcal R(M)$. The manifold of curvature radii is $4$-dimensional, hence there exists a linear combination of $X_1,\dots,X_5$ vanishing at $(x,R)$, and thus vanishing everywhere. We deduce that $\dim\mathcal L\leq 4$. Assume now that $\dim\mathcal L=4$ and let $X_1,\dots, X_4$ be a basis for $\mathcal L$. By the above argument $X_1,\dots,X_4$ are linearly independent at every point of $\mathcal R(M)$.
    Thus they can be used to produce local coordinates in the neighbourhood of every point, by means of the map 
   \begin{equation*}
       (s_1,\dots,s_4)\mapsto e^{s_1X_1}\circ\dots\circ e^{s_4X_4}(x,R).
   \end{equation*}
   This implies that the structure coefficients of the frame $f_1,f_2,f_{12},f_{121}$ are constants (Theorem \ref{thm_D} implies that the latter is a basis for $T\mathcal R(M)$). In particular there exist real constants $c_1,c_2,c_{12},c_{121}$ such that 
   \begin{equation*}
       f_{1121}=[f_1,f_{121}]=c_1f_1+c_2f_2+c_{12}f_{12}+c_{121}f_{121}.
   \end{equation*}
   Now we exploit Proposition \ref{prop_riemann}, telling us that $c_1(x,R)=|R|^2 K(x)$, where $K(x)$ is the Gaussian curvature of $(M,g)$. Such a function is constant on a fixed fiber of the manifold of curvature radii $\mathcal R(M)=TM\setminus s_0$ if and only if it is identically zero. It follows that $K=0$.
\end{proof}
	\section{Similarity transformations of the plane}\label{secR2}
	In this section we show that the sub-Riemannian manifolds $\mathcal R_{a,b}(\mathbb R^2,g_e)$, where $g_e$ is the standard Euclidean metric, are isomorphic to left invariant sub-Riemannian structures on the group of orientation preserving homotheties of $\mathbb R^2$, which we denote with $G$.
	Moreover we give a a characterization of the sub-Riemannian geodesics of $\mathcal R_{0,1}(\mathbb R^2,g_e)$ in terms of the euclidean curvature of their projections to a plane. Such characterization in terms of curvature is similar to the one of \textit{Euler elasticae} (\cite{Elasticae}), which are projection of normal extremal trajectories of the nilpotent Engel group (\cite{Sach}). All the results of this section follow from straightforward computations, therefore the proofs are omitted.
	Let $(y_1,y_2,R_1,R_2)$ be global coordinates for $T\mathbb R^2\setminus\{0\}=\mathbb R^2\times \mathbb R^2\setminus\{0\}$.
	It is convenient to define a new set of coordinates $(\theta, r, x_1,x_2)$ as 
	\begin{equation}\label{eq_circ_coor}
		(y_1,y_2,R_1,R_2)=:(
		x_1+r\text{cos}\theta, x_2+r\text{sin}\theta,
		-r\text{cos}\theta,
		-r\text{sin}\theta).
	\end{equation}
	\paragraph{}In coordinates $(y,R)$ a point in $\mathcal R_{a,b}(\mathbb R^2,g_e)$ is interpreted as the curvature radius of some curve going through $y$. In the new coordinates $(\theta, r, x_1,x_2)$ we interpret a point in $\mathcal R_{a,b}(\mathbb R^2,g_e)$ as the osculating circle of such curve, having radius $(r\cos\theta,r\sin\theta)$
	and center $(x_1,x_2)$. 
	Observe that the data of a homothetic transformation, which in the case of $(\mathbb R^2, g_e)$ consists in a composition of rotations, dialations and translation, can be encoded into an osculating circle, i.e. a point of $\mathcal R_{a,b}(\mathbb R^2,g_e)$: given a circle $(\theta, r, x_1,x_2)$ we can dilate by $r$, rotate by $\theta$ and translate by $(x_1,x_2)$. 
	We have a diffeomorphism
	\begin{equation}\label{Qdiffeo}
		\begin{aligned}
			F:\mathcal R(\mathbb R^2)&\to G\\
			(\theta,r,x_1,x_2)&\mapsto Q:=\begin{pmatrix}
				r\text{cos}\theta& -r\text{sin}\theta & x_1\\
				r\text{sin}\theta& r\text{cos}\theta & x_2\\
				0 & 0& 1
			\end{pmatrix},
		\end{aligned}
	\end{equation}
	which allows us to push forward the sub-Riemannian structure of $\mathcal R_{a,b}(\mathbb R^2,g_e)$, to $G$. 
	The resulting sub-Riemannian structure is left invariant.
	\begin{prop}
		The frame $(f_1,f_2)$ can be written in coordinates $(\theta, r, x_1,x_2)$ as
		\begin{equation}\label{eq_vector_field_1}
			\begin{aligned}
				f_1=-\frac{\partial}{\partial\theta},\,\,\,\,   
				f_2=r\frac{\partial}{\partial r}
				-r\cos\theta\frac{\partial}{\partial x_1}
				-r\,sin\theta\frac{\partial}{\partial x_2}.
			\end{aligned}
		\end{equation}
		Both of these vector fields are pushforwarded to left invariant vector fields by the map \eqref{Qdiffeo}:
		\begin{equation}\label{eq_vector_field_2}
			\begin{aligned}
				&(F_\star f_1)(Q)=-Q\begin{pmatrix}
					0 & -1 & 0\\
					1 & 0 & 0\\
					0 & 0 & 0
				\end{pmatrix},\,\,\,\,
				&(F_\star f_2)(Q)=Q\begin{pmatrix}
					1 & 0 & -1\\
					0 & 1 & 0\\
					0 & 0 & 0
				\end{pmatrix}.
			\end{aligned}
		\end{equation}
	\end{prop}
    \begin{remark}
    The geometry of $G$ is  reminiscent of the left invariant sub-Riemannian structure on the group of rigid motions of the plane, related to ‘bicycling mathematics’, which has been studied in \cite{bicycle1}, \cite{bicycle2}, \cite{bicycle3}, \cite{bicycle4}. The group of rigid motions of $\mathbb R^2$ can be described as 
    \begin{equation*}
        SE_2=\left\{ \begin{pmatrix}
				\cos\theta& -\sin\theta & x_1\\
				\sin\theta& \cos\theta & x_2\\
				0 & 0& 1
			\end{pmatrix}\,:\, \theta\in[0,2\pi],\,(x_1,x_2)\in\mathbb R^2\,
        \right\}.
    \end{equation*}
    In coordinates $(\theta, x_1,x_2)$ we can define a left-invaraint sub-Riemann structure on $SE_2$ by declaring the fields 
    \begin{equation}
        X_1=\cos\theta\frac{\partial}{\partial x_1}+\sin\theta\frac{\partial}{\partial x_2},\,\,\,X_2=\frac{\partial}{\partial\theta},
    \end{equation}
    an orthonormal generating family.
    There exists a submersion
    \begin{equation}
        \begin{aligned}
        P:G&\to SE_2\\
        (\theta,r,x_1,x_2)&\mapsto (\theta,x_1,x_2),
         \end{aligned}
    \end{equation}
    satisfying 
    \begin{equation}
        P_\star\frac{1}{r}f_2=-X_1,\,\,\, P_\star f_1=-X_2.
    \end{equation}
    \end{remark}
    Let $h_1,h_2:T^\star\mathcal R(\mathbb R^2)\to \mathbb R$ be the Hamiltonian functions associated with $f_1,f_2$ and let $p_\theta,p_r,p_{x_1},p_{x_2}$ be the canonical momenta associated with the coordinates $\theta,r,x_1,x_2$, then 
    \begin{equation*}
        \begin{aligned}
            &h_1=-p_\theta,\\
            &h_2=rp_r-r\cos\theta p_{x_1}-r\sin\theta p_{x_2}.
        \end{aligned}
    \end{equation*}
    The sub-Riemannian Hamiltonian of $\mathcal R_{0,1}(\mathbb R^2,g_e)$ reads 
    \begin{equation}
        2H=h_1^2+h_2^2=p_\theta^2+(rp_r-r\cos\theta p_{x_1}-r\sin\theta p_{x_2})^2.
    \end{equation}
    If we make the change of coordinates $\rho=\log r$ the Hamiltonian becomes
    \begin{equation}\label{sub_Ham}
        2H=p_\theta^2+(p_\rho-e^\rho\cos\theta p_{x_1}-e^\rho\sin\theta p_{x_2})^2.
    \end{equation}
    The next proposition characterizes normal extremal trajectories in terms of the euclidean curvature of their projection to the $(\rho,\theta)$-plane.
    \begin{prop}
    The following quantities 
    \begin{equation*}
        \epsilon=\sqrt{p_{x_1}^2+p_{x_2}^2},\,\,\, \alpha=\arg{(p_{x_1}+ip_{x_2})},
    \end{equation*}
    are first integrals of the sub-Riemannian Hamiltonian \eqref{sub_Ham}. \newline
    A curve $(\rho,\theta):[0,1]\to\mathbb R^2$ is the projection of a normal extremal trajectory $(\rho,\theta,x_1,x_2):[0,1]\to\mathbb R^4$, if and only if its Euclidean curvature $\kappa$ satisfies 
    \begin{equation*}
        \kappa(\rho,\theta)=\epsilon e^\rho\sin(\theta-\alpha).
    \end{equation*}
    There are no strictly abnormal extremals.
    \end{prop}
    \section*{Acknowledgments}
    I would like to thank my supervisors Andrei Agrachev and Luca Rizzi. 
    The main idea of this paper comes from the model space described in Section \ref{secR2}, which was discovered by Adrei Agrachev. 
    Many of the proofs have been substantially improved thanks to the suggestions of Luca Rizzi. I would also like to thank the anonymous reviewer for asking the question that led to Theorem \ref{flat_thm}.
\printbibliography
\end{document}